\documentclass[11pt]{article}
\usepackage{enumerate}
\usepackage{amssymb,a4wide,latexsym,makeidx,epsfig,fleqn}
\usepackage{amsthm}
\usepackage{amsmath}
\usepackage{enumerate}
\usepackage{extarrows}
\usepackage{graphicx}
\usepackage{subfigure}
\usepackage{float}
\usepackage[justification=centering]{caption}
\newtheorem{theorem}{Theorem}[section]

\newtheorem{lemma}[theorem]{Lemma}

\newtheorem{corollary}[theorem]{Corollary}

\begin{document}
\textwidth 150mm \textheight 225mm
\title{Ordering digraphs with maximum outdegrees by their $A_{\alpha}$ spectral radius
\thanks{Supported by the National Natural Science Foundation of China (Nos. 12001434 and 12271439) and the Natural Science Basic Research Program of Shaanxi Province (No. 2022JM-006 and 2024JC-YBQN-0015).}}
\author{Zengzhao Xu$^{a}$, Weige Xi$^{a}$\footnote{Corresponding author.}, Ligong Wang$^{b}$\\
{\small $^{a}$ College of Science, Northwest A\&F University, Yangling, Shaanxi 712100, P.R. China}\\
{\small $^{b}$ School of Mathematics and Statistics, Northwestern Polytechnical University,}\\
{\small  Xi'an, Shaanxi 710129, P.R. China}\\
{\small E-mail: xuzz0130@163.com; xiyanxwg@163.com; lgwangmath@163.com}\\}
\date{}
\maketitle
\begin{center}
\begin{minipage}{120mm}
\vskip 0.3cm
\begin{center}
{\small {\bf Abstract}}
\end{center}
{\small  Let $G$ be a strongly connected digraph with $n$ vertices and $m$ arcs. For any real $\alpha\in[0,1]$, the $A_\alpha$ matrix of a digraph $G$ is defined as
$$A_\alpha(G)=\alpha D(G)+(1-\alpha)A(G),$$
where $A(G)$ is the adjacency matrix of $G$ and $D(G)$ is the outdegrees diagonal matrix of $G$.
The eigenvalue of $A_\alpha(G)$ with the largest modulus is called the $A_\alpha$ spectral radius of $G$, denoted by $\lambda_{\alpha}(G)$.
In this paper, we first obtain an upper bound on $\lambda_{\alpha}(G)$ for $\alpha\in[\frac{1}{2},1)$. Employing this upper bound, we prove that for two strongly connected digraphs $G_1$ and $G_2$ with $n\ge4$ vertices and $m$ arcs, and $\alpha\in [\frac{1}{\sqrt{2}},1)$, if the maximum outdegree $\Delta^+(G_1)\ge 2\alpha(1-\alpha)(m-n+1)+2\alpha$ and $\Delta^+(G_1)>\Delta^+(G_2)$, then $\lambda_\alpha(G_1)>\lambda_\alpha(G_2)$. Moreover, We also give another upper bound on $\lambda_{\alpha}(G)$ for $\alpha\in[\frac{1}{2},1)$. Employing this upper bound, we prove that for two strongly connected digraphs with $m$ arcs, and $\alpha\in[\frac{1}{2},1)$, if the maximum outdegree $\Delta^+(G_1)>\frac{2m}{3}+1$ and $\Delta^+(G_1)>\Delta^+(G_2)$, then $\lambda_\alpha(G_1)+\frac{1}{4}>\lambda_\alpha(G_2)$.

}
\end{minipage}
\end{center}

\section{Introduction }

Before delving into our research content, we will first provide a brief introduction. Our research results are based on considering the ancient problem of ordering graphs with various properties by their spectra, which can be traced back to the research conducted in 1981. The introduction section of our paper will be divided into two parts.

{\bf 1.1. Terminology, notation and related work}

Let $G=(V(G),E(G))$ be a digraph consists of vertex set $V(G)=\{v_1,v_2,\cdots,v_n\}$ and arc set $E(G)$. For an arc $(v_i,v_j)\in E(G)$, we call the first vertex $v_i$ is its tail and the second $v_j$ is its head, respectively. A digraph is simple if it has no loops and multiple arcs. For every pair of vertices $v_i, v_j\in V(G)$ in a digraph $G$, if there always exists a directed path from $v_i$ to $v_j$, we call $G$ is strongly connected. Throughout this paper all digraphs considered are simple and strongly connected.

For a simple strongly connected digraph $G$, we use $N^+(v_i)=\{v_j\in V(G)|(v_i,v_j)\in E(G)\}$ to denote the out-neighbors of $v_i$ in the digraph $G$. Let $N^+[v_i]$ denote the set $N^+(v_i)\cup \{v_i\}$. The outdegree of a vertex $v_i$ in the digraph $G$, denoted by $d_i^+$, is the number of arcs with tail $v_i$, i.e., $d_i^+=|N^+(v_i)|$. The maximum outdegree of $G$ is denoted by $\Delta^+(G)$, or simplify $\Delta^+$. Let $T_i^+=\sum\limits_{(v_i,v_j)\in E(G)}d_j^+$ denote the 2-outdegree of vertex $v_i$ and $m_i^+=\frac{T_i^+}{d_i^+}$ denote the average 2-outdegree of $v_i$. For other terminology and notations of digraphs, please refer to see \cite{BoMu,BG}.

For a strongly connected digraph $G$ with vertex set $V(G)=\{v_1,v_2,\ldots,v_n\}$ and arc set $E(G)$, the adjacency matrix of $G$ is defined as $A(G)=(a_{ij})_{n\times n}$, where $a_{ij}=1$ if $(v_i,v_j)\in E(G)$ and $a_{ij}=0$ otherwise. Let $D(G)=\textrm{diag}(d^+_1,d^+_2,\ldots,d^+_n)$ be the diagonal matrix with outdegrees of vertices of $G$. We call $Q(G)=D(G)+A(G)$ the signless Laplacian matrix of $G$. The eigenvalue of $Q(G)$ with the largest modulus is called the signless Laplacian spectral radius of $G$, denoted by $q(G)$. In 2019, Liu et al. \cite{LWCL} defined the $A_\alpha$ matrix of $G$ as
$$A_\alpha(G)=\alpha D(G)+(1-\alpha) A(G), \ \textrm{where} \  0\le\alpha\le1.$$
It is clear that
$$A_0(G)=A(G),\ \ \ 2A_\frac{1}{2}(G)=Q(G).$$
Therefore, the matrix $A_\alpha(G)$ extends both $A(G)$ and $Q(G)$. Thus, it is very important to study the matrix $A_\alpha(G)$ of $G$. The eigenvalue of $A_\alpha(G)$ with the largest modulus is called the $A_\alpha$ spectral radius of $G$, denoted by $\lambda_\alpha(G)$. For $\alpha\in[0,1)$, if $G$ is a strongly connected digraph, then $A_\alpha(G)$ is nonnegative irreducible. In this case, it follows from the Perron Frobenius Theorem \cite{HJ} that $\lambda_\alpha(G)$ is an eigenvalue of $A_\alpha(G)$, and there is a unique positive unit eigenvector corresponding to $\lambda_\alpha(G)$. In the case when $\alpha=1$, $A_1(G)=D(G)$, which makes no sense. Therefore, we only consider that $0\leq \alpha<1$ in the rest of this paper.

The studying on the spectral radius of graphs and digraphs is an important topic in the spectral graph theory. Recently, more and more attention has been paid to the spectral properties of $A_\alpha$ matrix of a digraph. For example, in \cite{XWW}, Xi et al. determined the digraphs which attain the maximum (or minimum) $A_\alpha$ spectral radius among all strongly connected digraphs with given parameters such as girth, clique number, vertex connectivity and arc connectivity. In \cite{XW}, Xi and Wang established some lower bounds on $\Delta^+(G)-\lambda_{\alpha}(G)$ for
strongly connected irregular digraph $G$ with given maximum outdegree and some
other parameters. In \cite{BGGA}, Baghipur et al. obtained some sharp bounds on the $A_\alpha$ spectral radius of digraphs in terms of some parameters such as the outdegrees, the maximum outdegree, the number of vertices, the number of arcs and the parameter $\alpha$. Some other works on the spectral
properties of $A_\alpha$ matrix of a digraph can be found in \cite{GB,XSLC,XW1,YBW,YLY}.

{\bf 1.2 Background motivation for the presented results.}

In 1981, Cvetkovi\'{c} \cite{CV} proposed twelve
directions for further research in the theory of graph spectra, one of which is classifying and ordering graphs. Since then, ordering graphs with various properties by their spectra, particularly by their largest eigenvalue, has become a popular research topic. At present, there is no simple and generality method for ordering graphs by their spectra. In order to find a solution to the problem, Liu et al. \cite{LLC} put forward an idea: whether can we transfer the comparison of spectra of two graphs to the comparison of some parameters of them? By this means, the comparison of spectra of two graphs become much easier. Then, Liu et al. \cite{LLC} proved that for two connected graphs $H_1$ and $H_2$ with $n$ vertices and $m$ edges, if the maximum degree $\Delta(H_1)\ge m-\frac{n-3}{2}$ and $\Delta(H_1)> \Delta(H_2)$, then the signless Laplacian spectral radii of $H_1$ and $H_2$ satisfy $q(H_1)>q(H_2)$. Guo and Zhang \cite{GZ} extended the conclusion of Liu et al. \cite{LLC} to $A_\alpha$ matrix of graphs, they proved that for two connected graphs $H_1$ and $H_2$ with $n$ vertices and $m$ edges, if $\alpha\in[\frac{1}{2},1)$, the maximum degree $\Delta(H_1)\ge 2\alpha(1-\alpha)(2m-n+1)+2\alpha$ and $\Delta(H_1)> \Delta(H_2)$, then the $A_\alpha$ spectral radii of $H_1$ and $H_2$ satisfy $\lambda_{\alpha}(H_1)>\lambda_{\alpha}(H_2)$. Later, Zhang and Guo \cite{ZG} gave a new conclusion of comparing the signless Laplacian spectral radius of graphs, which they proved that for two connected graphs $H_1$ and $H_2$ with size $m\ge4$, if the maximum degree $\Delta(H_1)\ge \frac{2m}{3}+1$ and $\Delta(H_1)> \Delta(H_2)$, then the signless Laplacian spectral radii of $H_1$ and $H_2$ satisfy $q(H_1)>q(H_2)$. Ye et al. \cite{YGZ} extended the conclusion of Guo and Zhang, they proved that for two connected $H_1$ and $H_2$ with $n$ vertices and $m$ edges, if $\alpha\in [\frac{1}{2},1)$, the maximum degree $\Delta(H_1)\ge\frac{2m}{3}+1$ and $\Delta(H_1)> \Delta(H_2)$, then the $A_\alpha$ spectral radii of $H_1$ and $H_2$ satisfy $\lambda_{\alpha}(H_1)>\lambda_{\alpha}(H_2)$. These are the classic papers on classifying and ordering undirected graphs.

Up to now, there are a few results on classifying and ordering digraphs. Motivated by the results obtained in \cite{GZ} and \cite{YGZ}, in this paper we study the problem of ordering digraphs with various properties by their related spectra radius. The main purpose of our paper is to extend some results to digraphs, and our main results are as follows.

\begin{theorem} Let $G_1$ and $G_2$ be two strongly connected digraphs with $n\ge4$ vertices and $m$ arcs. For $\alpha\in [\frac{1}{\sqrt{2}},1)$, if $\Delta^+(G_1)\ge 2\alpha(1-\alpha)(m-n+1)+2\alpha$ and $\Delta^+(G_1)>\Delta^+(G_2)$, then $\lambda_\alpha(G_1)>\lambda_\alpha(G_2)$.
\end{theorem}

\begin{theorem} Let $G_1$ and $G_2$ be two strongly connected digraphs with $m$ arcs. For $\alpha\in[\frac{1}{2},1)$, if $\Delta^+(G_1)>\frac{2m}{3}+1$ and $\Delta^+(G_1)>\Delta^+(G_2)$, then $\lambda_\alpha(G_1)+\frac{1}{4}>\lambda_\alpha(G_2)$.
\end{theorem}

The rest of this paper is organized as follows. In Section 2, we give some lemmas to prove the theorems in the following sections. In Section 3, we first present an upper bound on the $A_\alpha$ spectral radius of digraphs, and using this upper bound, we prove Theorem 1.1. In Section 4, we also present another upper bound on the $A_\alpha$ spectral radius of digraphs, and using this upper bound, we prove Theorem 1.2.

\section{Preliminaries}

In this section, we give some lemmas which will be used in the follows. The following observation can be found in \cite{XWW}.

\noindent\begin{lemma}\label{le:1}(\cite{XWW})
Let $G$ be a strongly connected digraph with the maximum outdegree $\Delta^+$. Then
$\lambda_\alpha(G)>\alpha\Delta^+$.
\end{lemma}

The following observation can be found in \cite{BGGA}, which is a generalization of the related result of undirected graphs.

\begin{lemma}\label{le:2}(\cite{BGGA}) \ Let $G=(V(G), E(G))$ be a strongly
connected digraph of order $n$ having outdegrees $d_1^+\ge d_2^+\ge\cdots\ge d_n^+$. Let $m_i^+$ be the average 2-outdegree of the vertex $v_i$, where $1\le i\le n$ and $\alpha\in[0,1)$. Then
\begin{equation}\label{eq:1}
	\min\{\alpha d_i^++(1-\alpha)m_i^+,v_i\in V(G)\}\le \lambda_\alpha(G)\le\max\{\alpha d_i^++(1-\alpha)m_i^+,v_i\in V(G)\}.
\end{equation}
Equality is attained on either side of (\ref{eq:1}) if and only if $\alpha d_i^++(1-\alpha)m_i^+$ is the same for all $1\le i\le n$.
\end{lemma}

\section{The proof of Theorem 1.1}

For a connected graph with $n\ge4$ vertices and $m$ edges, Guo and Zhang \cite{GZ} proved that
\begin{equation*}
	\lambda_{\alpha}(G)\le \max\left\{\alpha\Delta(G),(1-\alpha)(m-\frac{n-1}{2})\right\}+2\alpha.
\end{equation*}

In this paper, we extend this result to $A_{\alpha}$ spectral radius of a strongly connected digraph.

\begin{theorem}\label{T:1} Let $G$ be a strongly connected digraph with $n\ge4$ vertices and $m$ arcs. If $\alpha\in[\frac{1}{2},1)$, then
	\begin{equation*}
		\lambda_{\alpha}(G)\le \max \left\{\alpha\Delta^+(G),(1-\alpha)\frac{m-n+1}{2}\right\}+2\alpha.
	\end{equation*}
\end{theorem}

\begin{proof} To simplify the process, let $\Delta^+=\Delta^+(G)$, by Lemma \ref{le:2},
	\small{\begin{equation*}
	 \lambda_\alpha(G)\le\max\left\{\alpha d_i^++(1-\alpha)m_i^+,v_i\in V(G)\right\}=\max\left\{\alpha d_i^++(1-\alpha)\frac{\sum\limits_{v_j\in N^+(v_i)}d_j^+}{d_i^+},v_i\in V(G)\right\}.
	\end{equation*}}
Let $u$ be a vertex of $G$ that satisfies the condition
	 	\begin{equation*}
	 	\alpha d_u^++(1-\alpha)\frac{\sum\limits_{v_j\in N^+(u)}d_j^+}{d_u^+}=\max\left\{\alpha d_i^++(1-\alpha)\frac{\sum\limits_{v_j\in N^+(v_i)}d_j^+}{d_i^+},v_i\in V(G)\right\}.
	 \end{equation*}
Since $1\le d_u^+\le \Delta^+$, then we will discuss the proof in two ways according to the value of $d_u^+$.
	
If $d_u^+=1$, for $\alpha\in[\frac{1}{2},1)$, it is easy to see that
	$$\alpha+(1-\alpha)\Delta^+<\max \left\{\alpha\Delta^+,(1-\alpha)\frac{m-n+1}{2}\right\}+2\alpha.$$
Thus, by Lemma \ref{le:2},
$$\lambda_\alpha(G)<\max \left\{\alpha\Delta^+,(1-\alpha)\frac{m-n+1}{2}\right\}+2\alpha.$$

Based on the above result, we assume that $2\le d_u^+\le \Delta^+$ in the follows. Considering that
$$m=\sum\limits_{v_i\in V(G)}d_i^+=d_u^++\sum\limits_{v_i\in N^+(u)}d_i^++\sum\limits_{v_i\in V(G)-N^+[u]}d_i^+,$$
then
\begin{align*}
	\sum\limits_{v_i\in N^+(u)}d_i^+&=m-d_u^+-\sum\limits_{v_i\in V(G)-N^+[u]}d_i^+\\
	&\le m-d_u^+-(n-1-d_u^+)\\
	&=m-n+1,
\end{align*}
with the equality holds if and only if $d_i^+=1$ for $v_i\in V(G)-N^+[u]$.

Furthermore, we have
 $$\lambda_\alpha(G)\le\alpha d_u^++(1-\alpha)\frac{\sum\limits_{v_i\in N^+(u)}d_i^+}{d_u^+}\le\alpha d_u^++(1-\alpha)\frac{m-n+1}{d_u^+}.$$
Let
 $$h(x)=\alpha x+(1-\alpha)\frac{m-n+1}{x}.$$
For $\alpha\in[\frac{1}{2},1)$ and $x>0$, by a simple calculation, we get $$h''(x)=2\frac{1-\alpha}{x^3}(m-n+1)>0.$$
Thus, for $x>0$, the function $h(x)$ is convex and its maximum value is attained at one of the endpoints of any closed interval. Therefore, for $2\le d_u^+\le \Delta^+$,
\begin{align*}
	\lambda_\alpha(G)&\le\alpha d_u^++(1-\alpha)\frac{m-n+1}{d_u^+}\\
	&\le \max\left\{2\alpha+\frac{1-\alpha}{2}(m-n+1),\alpha\Delta^++\frac{1-\alpha}{\Delta^+}(m-n+1)\right\}.
\end{align*}
Hence, in order to obtain Theorem \ref{T:1}, we only need to prove the following inequality. $$\max\left\{2\alpha+\frac{1-\alpha}{2}(m-n+1),\alpha\Delta^++\frac{1-\alpha}{\Delta^+}(m-n+1)\right\}\le\max\left\{\alpha\Delta^+,\frac{1-\alpha}{2}(m-n+1)\right\}+2\alpha.$$

To simplify the calculation process, let $t=m-n+1$ and
$$f(x)=\alpha x+(1-\alpha)\frac{t}{x},x\in[2,\Delta^+].$$
Then we just need to prove
$$\max\left\{f(2),f(\Delta^+)\right\}\le\max\left\{\alpha\Delta^+,\frac{1-\alpha}{2}t\right\}+2\alpha=\max\left\{\alpha\Delta^++2\alpha,f(2)\right\}.$$

{\bf Case 1.} $f(2)\ge f(\Delta^+).$

In this case, we can easily verify that
$$\max\left\{f(2),f(\Delta^+)\right\}=f(2)\le\max\left\{\alpha\Delta^++2\alpha,f(2)\right\}.$$

{\bf Case 2.} $f(2)< f(\Delta^+).$

In this case, $\max\left\{f(2),f(\Delta^+)\right\}=f(\Delta^+)$
and $f(2)=2\alpha+\frac{1-\alpha}{2}t<\alpha\Delta^++\frac{1-\alpha}{\Delta^+}t=f(\Delta^+)$.
By simplifying this inequality, we can get $\alpha(\Delta^+)^2-(2\alpha+\frac{1-\alpha}{2}t)\Delta^++(1-\alpha)t>0.$
For the equation $\alpha x^2-(2\alpha+\frac{1-\alpha}{2}t)x+(1-\alpha)t=0$, it has two roots
$$x=\frac{2\alpha+\frac{1-\alpha}{2}t\pm \sqrt{(2\alpha-\frac{1-\alpha}{2}t)^2}}{2\alpha}.$$
Since $f(2)< f(\Delta^+)$, to prove that $\max\left\{f(2),f(\Delta^+)\right\}\le\max\left\{\alpha\Delta^++2\alpha,f(2)\right\}$, we only need to prove
$$f(\Delta^+)=\alpha\Delta^++\frac{1-\alpha}{\Delta^+}t\le\alpha\Delta^++2\alpha\Longleftrightarrow \frac{1-\alpha}{\Delta^+}t\le 2\alpha.$$
In the following, we consider the following two subcases by comparing the value of $2\alpha$ and $\frac{1-\alpha}{2}t$.

{\bf Case 2.1.} $2\alpha\ge\frac{1-\alpha}{2}t$.

In this case, we get $t\le \frac{4\alpha}{1-\alpha}$. Since $\Delta^+\ge2$, it is easy to get
$$\frac{1-\alpha}{\Delta^+}t\le\frac{1-\alpha}{2}t\le\frac{1-\alpha}{2}\cdot \frac{4\alpha}{1-\alpha}=2\alpha.$$

{\bf Case 2.2.} $2\alpha<\frac{1-\alpha}{2}t$.

In this case, the two roots of the equation $\alpha x^2-(2\alpha+\frac{1-\alpha}{2}t)x+(1-\alpha)t=0$ are
$$x_1=\frac{1-\alpha}{2\alpha}t,\quad x_2=2.$$
Since $2\alpha<\frac{1-\alpha}{2}t$, we have $\frac{1-\alpha}{2\alpha}t>2, i.e., x_1>x_2.$
Since $\Delta^+\ge 2=x_2$ and $\Delta^+$ satisfies $\alpha(\Delta^+)^2-(2\alpha+\frac{1-\alpha}{2}t)\Delta^++(1-\alpha)t>0$, combined with the characteristics of quadratic function, we can conclude that $\Delta^+>\frac{1-\alpha}{2\alpha}t$. Thus $\frac{1-\alpha}{\Delta^+}t< 2\alpha$ and
$$f(\Delta^+)=\alpha\Delta^++\frac{1-\alpha}{\Delta^+}t<\alpha\Delta^++2\alpha\le\max\left\{\alpha\Delta^++2\alpha,f(2)\right\}.$$

Based on Case 1 and Case 2, we can finally get
$$\max\left\{2\alpha+\frac{1-\alpha}{2}(m-n+1),\alpha\Delta^++\frac{1-\alpha}{\Delta^+}(m-n+1)\right\}\le\max\left\{\alpha\Delta^+,\frac{1-\alpha}{2}(m-n+1)\right\}+2\alpha.$$
Therefore
$$\lambda_{\alpha}(G)\le \max\left\{\alpha\Delta^+(G),(1-\alpha)\frac{m-n+1}{2}\right\}+2\alpha,$$
this completes the proof.
\end{proof}

Taking $\alpha=\frac{1}{2}$, we can obtain an upper bound on the signless Laplacian spectral radius of a digraph.

\noindent\begin{corollary} Let $G$ be a strongly connected digraph with $n\ge4$ vertices and $m$ arcs. Let $q(G)$ be the signless Laplacian spectral radius of $G$, then
	\begin{equation*}
		q(G)\le \max \left\{\Delta^+(G),\frac{m-n+1}{2}\right\}+2.
	\end{equation*}
\end{corollary}

By using Theorem \ref{T:1}, we can prove Theorem 1.1.

{\bf Proof of Theorem 1.1:}

\begin{proof}

Let $\Delta^+_1=\Delta^+(G_1)$ and $\Delta^+_2=\Delta^+(G_2)$. Using Lemma \ref{le:2}, we get that
\begin{equation*}
	\lambda_\alpha(G_2)\le\max\left\{\alpha d_i^++(1-\alpha)m_i^+,v_i\in V(G_2)\right\}.
\end{equation*}
Let $u$ be a vertex of $G_2$ that satisfies the condition
\begin{equation*}
	\alpha d_u^++(1-\alpha)m_u^+=\max\left\{\alpha d_i^++(1-\alpha)m_i^+,v_i\in V(G_2)\right\}.
\end{equation*}
Since $1\le d_u^+\le \Delta^+$, then we will discuss the proof in two ways according to the value of $d_u^+$.

If $d_u^+=1$, for $\alpha\in[\frac{1}{\sqrt{2}},1)$, it is easy to see
$$\lambda_\alpha(G_2)\le\alpha+(1-\alpha)\Delta_2^+.$$
Since $\Delta_1^+\ge\Delta_2^++1$, we can conclude
\begin{align*}
	\lambda_\alpha(G_1)&>\alpha\Delta_1^+\ge \alpha(\Delta_2^++1)=\alpha\Delta_2^++\alpha\ge (1-\alpha)\Delta_2^++\alpha\ge \lambda_\alpha(G_2).
\end{align*}

If $2\le d_u^+\le \Delta_2^+$, by the proof of Theorem \ref{T:1}, we get
\begin{equation}\label{eq2}
	\lambda_\alpha(G_2)\le \max\left\{2\alpha+\frac{1-\alpha}{2}(m-n+1),\alpha\Delta_2^++\frac{1-\alpha}{\Delta_2^+}(m-n+1)\right\}.
\end{equation}

Next, we will prove that $\lambda_{\alpha}(G_1) > \lambda_\alpha(G_2)$ in the following cases.

{\bf Case 1.} $\Delta_2^+\ge 2\alpha(1-\alpha)(m-n+1)$.

Note that
\begin{align*}
	& \quad \quad \quad \alpha\Delta_2^++\frac{1-\alpha}{\Delta_2^+}(m-n+1)\ge 2\alpha+\frac{1-\alpha}{2}(m-n+1)\\
	&\Longleftrightarrow 2\alpha(\Delta_2^+)^2-(4\alpha+(1-\alpha)(m-n+1))\Delta_2^++2(1-\alpha)(m-n+1)\ge0\\
	&\Longleftrightarrow [2\alpha\Delta_2^+-(1-\alpha)(m-n+1)](\Delta_2^+-2)\ge0\\
	&\Longleftrightarrow\Delta_2^+\ge\frac{1-\alpha}{2\alpha}(m-n+1).
\end{align*}
Since $\alpha\in[\frac{1}{\sqrt{2}},1)$, we have
 $$\Delta_2^+\ge2\alpha(1-\alpha)(m-n+1)> \frac{1-\alpha}{2\alpha}(m-n+1).$$
Thus, one can easy see that $$\alpha\Delta_2^++\frac{1-\alpha}{\Delta_2^+}(m-n+1)\ge 2\alpha+\frac{1-\alpha}{2}(m-n+1).$$
Furthermore, it follows that $$\lambda_\alpha(G_2)\le \alpha\Delta_2^++\frac{1-\alpha}{\Delta_2^+}(m-n+1).$$
Since $\Delta_2^+\ge2\alpha(1-\alpha)(m-n+1)$ and $\Delta_1^+>\Delta_2^+$,
 \begin{align*}
 	\lambda_{\alpha}(G_1)-\lambda_{\alpha}(G_2)&>\alpha\Delta_1^+-\alpha\Delta_2^+-\frac{1-\alpha}{\Delta_2^+}(m-n+1)\\
 	&=\alpha(\Delta_1^+-\Delta_2^+)-\frac{1-\alpha}{\Delta_2^+}(m-n+1)\\
 	&\ge\alpha(\Delta_1^+-\Delta_2^+)-\frac{1-\alpha}{2\alpha(1-\alpha)(m-n+1)}(m-n+1)\\
 	&\ge \alpha-\frac{1}{2\alpha}.
 \end{align*}
Note that $\alpha\ge \frac{1}{\sqrt{2}}$, which implies that $\lambda_{\alpha}(G_1)>\lambda_{\alpha}(G_2)$.

{\bf Case 2.} $\Delta_1^+-\Delta_2^+\ge2\alpha(1-\alpha)(m-n+1)$.

Since $\Delta_2^+\ge2$, $\Delta_1^+-2\ge \Delta_1^+-\Delta_2^+$, through simple calculation, we can easily prove that
\begin{align*}
	 \alpha\Delta_1^+-2\alpha-\frac{1-\alpha}{2}(m-n+1) &\ge\alpha(\Delta_1^+-\Delta_2^+)-\frac{1-\alpha}{2}(m-n+1)\\
	&\ge 2\alpha^2(1-\alpha)(m-n+1)-\frac{1-\alpha}{2}(m-n+1)\\
	&=\frac{1-\alpha}{2}(4\alpha^2-1)(m-n+1)>0,
\end{align*}
and
\begin{align*}
	\alpha\Delta_1^+-\alpha\Delta_2^+-\frac{1-\alpha}{\Delta_2^+}(m-n+1) &\ge\alpha(\Delta_1^+-\Delta_2^+)-\frac{1-\alpha}{2}(m-n+1)\\
	&\ge 2\alpha^2(1-\alpha)(m-n+1)-\frac{1-\alpha}{2}(m-n+1)\\
	&=\frac{1-\alpha}{2}(4\alpha^2-1)(m-n+1)>0.
\end{align*}
By (\ref{eq2}) and Lemma \ref{le:1}, we obtain
 $\lambda_{\alpha}(G_1)>\alpha\Delta_1^+>\lambda_{\alpha}(G_2).$

{\bf Case 3.} $\Delta_2^+<2\alpha(1-\alpha)(m-n+1)$ and $\Delta_1^+-\Delta_2^+<2\alpha(1-\alpha)(m-n+1).$

Using Theorem \ref{T:1}, we get
$$\lambda_{\alpha}(G_2)\le \max\left\{\alpha\Delta^+(G_2),(1-\alpha)\frac{m-n+1}{2}\right\}+2\alpha.$$
For $\alpha\in[\frac{1}{\sqrt{2}},1)$, since $\Delta_1^+\ge 2\alpha(1-\alpha)(m-n+1)+2\alpha$, we have
$$\Delta_1^+-\Delta_2^+>2\alpha(1-\alpha)(m-n+1)+2\alpha-2\alpha(1-\alpha)(m-n+1)=2\alpha\ge\sqrt{2}>1.$$
Hence, $\Delta_1^+-\Delta_2^+\ge2$ and
$$\alpha\Delta_1^+-\alpha\Delta_2^+-2\alpha=\alpha(\Delta_1^+-\Delta_2^+)-2\alpha\ge0.$$
On the other hand, $\Delta_1^+\ge 2+\Delta_2^+\ge4$,
$$m=\sum_{i=1}^{n}d_i^+(G_1)\ge 4+(n-1)=n+3.$$
Thus, $m-n+1\ge 4$ and
\begin{align*}
	\alpha\Delta_1^+-\frac{1-\alpha}{2}(m-n+1)-2\alpha &\ge2\alpha^2(1-\alpha)(m-n+1)+2\alpha^2-\frac{1-\alpha}{2}(m-n+1)-2\alpha\\
	&=2\alpha^2(1-\alpha)(m-n+1)-2\alpha(1-\alpha)-\frac{1-\alpha}{2}(m-n+1)\\
	&=\frac{1-\alpha}{2}[4\alpha^2(m-n+1)-4\alpha-(m-n+1)]\\
	&\ge\frac{1-\alpha}{2}[2(m-n+1)-4\alpha-(m-n+1)]\\
	&\ge \frac{1-\alpha}{2}(4-4\alpha)>0.
\end{align*}
Therefore, $\lambda_{\alpha}(G_1)>\alpha\Delta_1^+>\lambda_{\alpha}(G_2)$.

Based on the above process, we have ultimately completed the proof of Theorem 1.1.
\end{proof}

{\bf Remark:} For $\alpha\in [\frac{1}{2},\frac{1}{\sqrt{2}})$, the method used in the proof of Theorem 1.1 may not be true. Because in Case 1, we have $\lambda_{\alpha}(G_1)-\lambda_{\alpha}(G_2)\ge \alpha-\frac{1}{2\alpha}$ and $\alpha-\frac{1}{2\alpha}<0$ under this condition. Thus, we may not be able to get $\lambda_{\alpha}(G_1)>\lambda_{\alpha}(G_2)$. Therefore, we pose a question: whether is there a new method that can lead to the conclusion holding at $\alpha\in [\frac{1}{2},\frac{1}{\sqrt{2}})$?

\section{The proof of Theorem 1.2}

Let $k\ge1$ be an integer, for a connected graph $H$ with fixed size $m$ and maximum degree $\Delta\le m-k$. If $m\ge3k$ and $\alpha\in[\frac{1}{2},1)$, Ye et al. \cite{YGZ} proved that
\begin{equation*}
	\lambda_{\alpha}(H)\le \alpha(m-k)+\frac{2(1-\alpha)k}{m-k}+1-\alpha.
\end{equation*}

In this paper, we extend this result to $A_{\alpha}$ spectral radius of a strongly connected digraph.

\begin{theorem}\label{T:3} Let $k\ge1$ be an integer and $G$ be a strongly connected digraph with $m$ arcs and maximum outdegree $\Delta^+(G)\le m-k$. If $m\ge3k$ and $\alpha\in[\frac{1}{2},1)$, then
	\begin{equation}\label{eq:10}
		\lambda_{\alpha}(G)\le \alpha(m-k)+\frac{(1-\alpha)k}{m-k}+1-\alpha.
	\end{equation}
\end{theorem}

\begin{proof}
Let $k\ge1$ be an integer and $G$ be a strongly connected digraph with $m$ arcs and maximum outdegree $\Delta^+(G)\le m-k$. Let $\Delta^+=\Delta^+(G)$, using Lemma \ref{le:2}, we get
	\small{\begin{equation*}
		\lambda_\alpha(G)\le\max\left\{\alpha d_i^++(1-\alpha)m_i^+,v_i\in V(G)\right\}=\max\left\{\alpha d_i^++(1-\alpha)\frac{\sum\limits_{v_j\in N^+(v_i)}d_j^+}{d_i^+},v_i\in V(G)\right\}.
	\end{equation*}}
	
If $m=3$ or 4, nothing that $m\ge3k$, we can infer that $k=1$. In the case when $m=3$, then $G=G_1$($G_1$ is a digraph as shown in Figure 1) and
	$$\alpha(m-k)+\frac{(1-\alpha)k}{m-k}+1-\alpha=\frac{1}{2}\alpha+\frac{3}{2}.$$
By Lemma \ref{le:2}, we can easily verify that $\lambda_{\alpha}(G)\le\alpha+1-\alpha=1<\frac{1}{2}\alpha+\frac{3}{2}$.
\begin{figure}[H]
	\begin{centering}
		\includegraphics[scale=0.5]{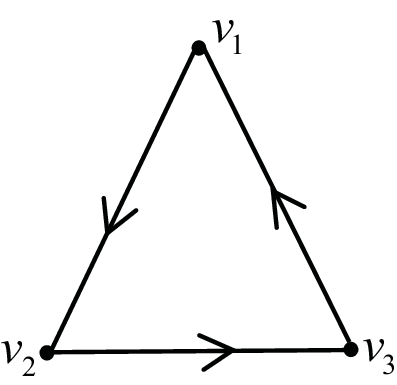}
		\caption{The digraph $G_1$.}\label{Fig.1.}
	\end{centering}
\end{figure}
In the case when $m=4$, the digraphs that satisfy the condition are shown in Figure 2. Using the same way, we can verify that (\ref{eq:10}) holds.
	\begin{figure}[H]
	\begin{centering}
		\includegraphics[scale=0.5]{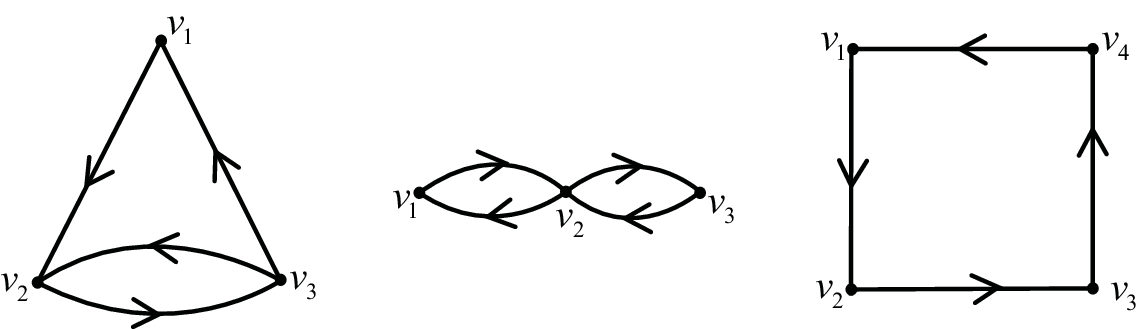}
		\caption{The digraphs with size $m=4$.}\label{Fig.2.}
	\end{centering}
\end{figure}
Based on the above results, we just consider the case when $m\ge5$ in the follows.
	
Let $u$ be a vertex of $G$ that satisfies the condition
	\begin{equation*}
		\alpha d_u^++(1-\alpha)\frac{\sum\limits_{v_j\in N^+(u)}d_j^+}{d_u^+}=\max\left\{\alpha d_i^++(1-\alpha)\frac{\sum\limits_{v_j\in N^+(v_i)}d_j^+}{d_i^+},v_i\in V(G)\right\}.
	\end{equation*}
Since $1\le d_u^+\le \Delta^+\le m-k$, then we will discuss the proof in two ways according to the value of $d_u^+$.
	
In the case when $d_u^+=1$, for $\alpha\in[\frac{1}{2},1)$ and $m\ge3k$, it is easy to see that
	$$\alpha+(1-\alpha)\Delta^+\le\alpha+(1-\alpha)(m-k)<\alpha(m-k)+\frac{(1-\alpha)k}{m-k}+1-\alpha.$$
Thus, by Lemma \ref{le:2},
	$$\lambda_\alpha(G)<\alpha(m-k)+\frac{(1-\alpha)k}{m-k}+1-\alpha.$$
	
In the case when $d_u^+=2$, for $\alpha\in[\frac{1}{2},1)$ and $m\ge3k$, then by Lemma \ref{le:2},
	$$\lambda_{\alpha}(G)<2\alpha+(1-\alpha)\frac{m}{2}.$$
Moreover,
	$$\alpha(m-k)+\frac{(1-\alpha)k}{m-k}+1-\alpha-[2\alpha+(1-\alpha)\frac{m}{2}]=\frac{(2+k-m)(2\alpha k+m-3\alpha m)}{2(m-k)}.$$
Note that $2+k-m\le 2-2k\le0$, then we will check that $2\alpha k+m-3\alpha m<0\Longleftrightarrow \alpha>\frac{m}{3m-2k}$.
Since $m\ge3k$ and $\alpha\in[\frac{1}{2},1)$, we can easy to see that $\alpha\ge\frac{1}{2}>\frac{m}{3m-2k}$.
Hence,
	$$\lambda_\alpha(G)<2\alpha+(1-\alpha)\frac{m}{2}<\alpha(m-k)+\frac{(1-\alpha)k}{m-k}+1-\alpha.$$
	
Based on the above results, we assume that $3\le d_u^+\le \Delta^+\le m-k$ in the follows. Noting that
	$$m=\sum\limits_{v_i\in V(G)}d_i^+=d_u^++\sum\limits_{v_i\in N^+(u)}d_i^++\sum\limits_{v_i\in V(G)-N^+[u]}d_i^+.$$
Thus, one can easily prove that
	\begin{align*}
		\sum\limits_{v_i\in N^+(u)}d_i^+&=m-d_u^+-\sum\limits_{v_i\in V(G)-N^+[u]}d_i^+\le m-d_u^+,
	\end{align*}
where the equality holds if and only if $V(G)-N^+[u]=\emptyset$, which implies $V(G)=\{u\}\cup N^+(u)$. Therefore,
	$$\lambda_\alpha(G)\le\alpha d_u^++(1-\alpha)\frac{\sum\limits_{v_i\in N^+(u)}d_i^+}{d_u^+}\le\alpha d_u^++(1-\alpha)\frac{m-d_u^+}{d_u^+}.$$
Let
	$$h(x)=\alpha x+(1-\alpha)\frac{m}{x}.$$
For $\alpha\in[\frac{1}{2},1)$ and $x>0$, by a simple calculation, we get
     $$h'(x)=\alpha-(1-\alpha)\frac{m}{x^2}.$$
Hence $h(x)$ is increasing with respect to $x\in(\sqrt{\frac{1-\alpha}{\alpha}m},+\infty)$ and decreasing with respect to $x\in(0,\sqrt{\frac{1-\alpha}{\alpha}m})$.
Nothing that
	 $$h''(x)=2\frac{1-\alpha}{x^3}m>0.$$
Thus, the function $h(x)$ is convex for $x>0$ and its maximum value is attained at one of the endpoints of any closed interval. Therefore, in the case when $3\le d_u^+\le \Delta^+$, we get
	\begin{align*}
		\lambda_\alpha(G)&\le\alpha d_u^++(1-\alpha)\frac{m}{d_u^+}+\alpha-1\\
		&\le \max\left\{h(3),h(\Delta)\right\}+\alpha-1.
	\end{align*}	
	
	{\bf Case 1.} $\Delta^+\ge \max\{3,\frac{1-\alpha}{3\alpha}m\}$.
	
In this case, noting that the function $h(x)$ is convex for $x>0$ and $h(3)=h(\frac{1-\alpha}{3\alpha}m)$, we have $h(3)\le h(\Delta^+)$. Thus,
	$$\lambda_{\alpha}(G)\le h(\Delta^+)+\alpha-1.$$
	
Next, we will compare 3 and $\frac{1-\alpha}{3\alpha}m$.
	
If $3\le\frac{1-\alpha}{3\alpha}m$, one can easily get that
	$$3\le \sqrt{\frac{1-\alpha}{\alpha}m}\le\frac{1-\alpha}{3\alpha}m\le\Delta^+\le m-k.$$
	
In the case when $3>\frac{1-\alpha}{3\alpha}m$, we have
		$$\sqrt{\frac{1-\alpha}{\alpha}m}<3\le\Delta^+\le m-k.$$
Therefore, based on the above process, we can always obtain the result that $h(\Delta^+)\le h(m-k)$.

Hence,
	$$\lambda_{\alpha}(G)\le h(m-k)+\alpha-1=\alpha(m-k)+\frac{(1-\alpha)k}{m-k}<\alpha(m-k)+\frac{(1-\alpha)k}{m-k}+1-\alpha.$$
	
	{\bf Case 2.} $\Delta^+< \max\{3,\frac{1-\alpha}{3\alpha}m\}$.
	
In this case, since $\Delta^+\ge3$, we have $3<\frac{1-\alpha}{3\alpha}m$.
Thus, $\Delta^+<\frac{1-\alpha}{3\alpha}m$ and
	$$3<\sqrt{\frac{1-\alpha}{\alpha}m}<\frac{1-\alpha}{3\alpha}m.$$
Therefore, $h(3)=h(\frac{1-\alpha}{3\alpha}m)\ge h(\Delta^+)$ and
	$$\lambda_{\alpha}(G)\le h(3)+\alpha-1.$$
Nothing that $m\ge 3k$ and $\alpha\in[\frac{1}{2},1)$, we can easily verify that $\frac{1-\alpha}{3\alpha}m<m-k$. Moreover, $h(3)=h(\frac{1-\alpha}{3\alpha}m)$ and $h(x)$ is increasing with respect to $x\in(\sqrt{\frac{1-\alpha}{\alpha}m},+\infty)$, then
	$$\lambda_{\alpha}(G)\le h(m-k)+\alpha-1=\alpha(m-k)+\frac{(1-\alpha)k}{m-k}<\alpha(m-k)+\frac{(1-\alpha)k}{m-k}+1-\alpha.$$
	
Therefore, combining by the above process, we can finally prove the conclusion
\begin{equation*}
	\lambda_{\alpha}(G)\le \alpha(m-k)+\frac{(1-\alpha)k}{m-k}+1-\alpha,
\end{equation*}
this completes the proof.
\end{proof}

Let $\alpha=\frac{1}{2}$, we can obtain an upper bound on the signless Laplacian spectral radius of a digraph.

\noindent\begin{corollary} Let $k\ge1$ be an integer and $G$ be a strongly connected digraph with $m$ arcs and maximum outdegree $\Delta^+(G)\le m-k$. Let $q(G)$ be the signless Laplacian spectral radius of $G$.
If $m\ge3k$, then
	\begin{equation*}
		q(G)\le m-k+1+\frac{k}{m-k}.
	\end{equation*}
\end{corollary}

By Theorem \ref{T:3}, we can prove Theorem 1.2.

{\bf Proof of Theorem 1.2:}

\begin{proof}
	
Let $\Delta^+_1=\Delta^+(G_1)=m-k_1$ and $\Delta^+_2=\Delta^+(G_2)=m-k_2$, where $0\le k_1,k_2<m$ are both integer.
	
Noting that $\Delta^+(G_1)>\Delta^+(G_2)$ and $\Delta^+(G_1)>\frac{2m}{3}+1$, one can easily verify that $k_1<k_2$ and $k_1\le \frac{1}{3}m-1.$
	
In the case when $k_1=0$, we have $\Delta_1^+=m$, which implies $G_1$ is a digraph as shown in Figure 3. This contradicts the fact that $G_1$ is strongly connected.
\begin{figure}[H]
	\begin{centering}
		\includegraphics[scale=0.5]{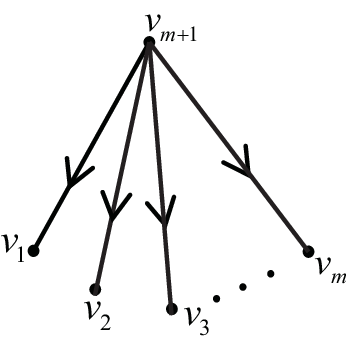}
		\caption{The digraph $G_1$.}\label{Fig.3.}
	\end{centering}
\end{figure}
	
In the case when $k_1\ge 1$, by Lemma \ref{le:1},
	$$\lambda_{\alpha}(G_1)>\alpha\Delta_1^+=\alpha(m-k_1).$$
For $m\ge 3k_2$, since $k_2>k_1$, then $m-k_2\le m-k_1-1$ and $\frac{k_2}{m-k_2}\le \frac{1}{2}$.
Thus, by Theorem \ref{T:3}, we can easily verify that
	\begin{align*}
		\lambda_{\alpha}(G_2)&\le \alpha(m-k_2)+\frac{1-\alpha}{m-k_2}k_2+1-\alpha\\
		&\le \alpha(m-k_1-1)+\frac{1}{2}(1-\alpha)+1-\alpha\\
		&\le  \alpha(m-k_1)+\frac{3}{2}-\frac{5}{2}\alpha\\
		&\le \alpha\Delta_1^++\frac{1}{4}\ \ \ (\Delta_1^+=m-k_1)\\
		&<\lambda_{\alpha}(G_1)+\frac{1}{4}.
	\end{align*}
In the case when $m<3k_2$, let $l=\lfloor \frac{1}{3}m \rfloor$. One can easily get that $m\ge 3l$ and $l\le k_2$. It follows that $m-k_2\le m-l$ and $\frac{l}{m-l}\le \frac{1}{2}$. Since $k_1\le \frac{1}{3}m-1$, it is easy to verify that $k_1+1\le l$. Thus, $m\ge 3l$ and $\Delta^+_2=m-k_2\le m-l$.
Then, by Theorem \ref{T:3} and Lemma \ref{le:1},
		\begin{align*}
		\lambda_{\alpha}(G_2)&\le \alpha(m-l)+\frac{1-\alpha}{m-l}l+1-\alpha\\
		&\le \alpha(m-k_1-1)+\frac{1}{2}(1-\alpha)+1-\alpha\\
		&\le  \alpha(m-k_1)+\frac{3}{2}-\frac{5}{2}\alpha\\
		&<\lambda_{\alpha}(G_1)+\frac{1}{4}.
	\end{align*}
Therefore, combining with the above process, we can finally prove the conclusion
	\begin{equation}
		\lambda_{\alpha}(G_2)<\lambda_{\alpha}(G_1)+\frac{1}{4},
	\end{equation}
this completes the proof.
\end{proof}

{\bf Remark:} In the proof of Theorem 1.2, we have $\lambda_{\alpha}(G_2)\le  \alpha(m-k_1)+\frac{3}{2}-\frac{5}{2}\alpha$. If $\alpha\in[\frac{3}{5},1)$, one can easily verify that $\frac{3}{2}-\frac{5}{2}\alpha\le0$, which implies that $\lambda_{\alpha}(G_2)\le \alpha(m-k_1)=\alpha\Delta_1^+<\lambda_{\alpha}(G_1)$. Therefore, we can get the following corollary:

\begin{corollary} Let $G_1$ and $G_2$ be two strongly connected digraphs with $m$ arcs. For $\alpha\in[\frac{3}{5},1)$, if $\Delta^+(G_1)>\Delta^+(G_2)$ and $\Delta^+(G_1)>\frac{2m}{3}+1$, then $\lambda_\alpha(G_1)>\lambda_\alpha(G_2).$
\end{corollary}

%\section*{Statements and Declarations}

%The authors declare that they have no conflict of interest.

\end{document}